\documentclass[11pt]{article}
\ifdefined\pdfinfoomitdate\pdfinfoomitdate=1\fi
\ifdefined\pdfsuppressptexinfo\pdfsuppressptexinfo=-1\fi
\ifdefined\pdftrailerid\pdftrailerid{}\fi
\usepackage[T1]{fontenc}
\usepackage{lmodern}
\usepackage[margin=1.10in]{geometry}
\usepackage{amsmath,amssymb,mathtools,amsthm}
\usepackage{enumitem}
\usepackage{microtype}
\usepackage{titlesec}
\usepackage{titling}
\usepackage{abstract}
\usepackage{xcolor}
\definecolor{referenceblue}{RGB}{0,65,140}
\usepackage[colorlinks=true,linkcolor=referenceblue,
  citecolor=referenceblue,urlcolor=referenceblue]{hyperref}
\hypersetup{
  pdftitle={The Higher-Dimensional Nitsche Conjecture: Sharp Bounds and Rigidity},
  pdfauthor={Bin Deng, Jiahuan Li, Yilu Liu, Xi-Nan Ma},
  pdfcreator={},pdfproducer={}}
\usepackage{bm}

\numberwithin{equation}{section}
\allowdisplaybreaks[2]
\setlist[enumerate]{topsep=0.45em,itemsep=0.15em,parsep=0pt}
\titleformat{\section}[block]
  {\centering\large\bfseries}
  {\thesection.}{0.70em}{}
\titleformat{name=\section,numberless}[block]
  {\centering\large\bfseries}
  {}{0pt}{}
\titlespacing*{\section}{0pt}{3.0ex plus 1ex minus .2ex}{1.5ex plus .2ex}

\titleformat{\subsection}[hang]
  {\normalsize\bfseries}
  {\thesubsection.}{0.60em}{}
\titleformat{name=\subsection,numberless}[block]
  {\centering\normalsize\bfseries}
  {}{0pt}{}
\titlespacing*{\subsection}{0pt}{2.25ex plus .7ex minus .2ex}{0.9ex plus .1ex}

\pretitle{\begin{center}\LARGE\bfseries}
\posttitle{\par\end{center}\vspace{0.45em}}
\preauthor{\begin{center}\normalsize}
\postauthor{\par\end{center}\vspace{-0.65em}}
\predate{}
\postdate{}

\newtheorem{theorem}{Theorem}[section]
\newtheorem{lemma}[theorem]{Lemma}

\newtheorem{corollary}[theorem]{Corollary}
\newtheorem{remark}[theorem]{Remark}

\newcommand{\A}{\mathbb A}
\newcommand{\Sph}{\mathbb S^{n-1}}
\newcommand{\dd}{\,\mathrm d}
\newcommand{\Kcal}{\mathcal K}
\newcommand{\Ncal}{\mathcal N}
\newcommand{\Tcal}{\mathcal T}
\newcommand{\Rplus}{R_{n,+}}
\newcommand{\Rminus}{R_{n,-}}
\newcommand{\Rstar}{R_*}
\newcommand{\ip}[2]{\langle #1,#2\rangle}

\title{The Higher-Dimensional Nitsche Conjecture:\\ Sharp Bounds and Rigidity}
\author{Bin Deng\qquad Jiahuan Li\qquad Yilu Liu\qquad Xi-Nan Ma}
\date{}

\begin{document}
\maketitle

\begin{abstract}
Let $n\ge3$ and let
$h:\A(r,1)\to\A(R,1)\subset\mathbb R^n$
be an onto homeomorphism with harmonic coordinate functions.  We prove the
sharp Nitsche bound
\[
 R\le R_{n,+}(r):=\frac{nr}{n-1+r^n},
\]
and, when $h$ interchanges the two ends, the strictly stronger sharp bound
\[
 R\le R_{n,-}(r):=\frac{nr^{n-1}}{1+(n-1)r^n}.
\]
In each case, equality holds only for the corresponding radial harmonic
homeomorphism, up to an orthogonal transformation.

Both bounds and their equality cases extend to bounded harmonic maps in
every dimension $n\ge2$.  This requires nonzero degree of $h/|h|$ on
interior spheres, $|h|\ge R$, and the corresponding $L^2$ limits of the
modulus at the two ends.  In the plane, we also prove the
nonzero-winding-number conjecture of Iwaniec--Kovalev--Onninen
[\emph{J. Amer. Math. Soc.} \textbf{24} (2011), 345--373], without
endpoint modulus conditions.

We prove a sharp $L^1$ contraction theorem for vector measures associated
with spherical couplings.  The first marginal is uniform, while the
second is unrestricted.  For bounded nonnegative nondecreasing zonal
kernels, the optimal constant is the degree-one multiplier.  If the
kernel is also strictly increasing, equality holds
exactly for couplings supported on orthogonal graphs.  This theorem is
the main tool in the proofs of the annular bounds and their rigidity.
\end{abstract}

\section{Introduction and main results}\label{sec:introduction}

\subsection{The Nitsche problem}
The Nitsche problem asks which concentric spherical annuli can be
joined by a harmonic homeomorphism.  Write
\[
 \A(r,1)=\{x\in\mathbb R^n:r<|x|<1\},
 \qquad 0<r<1.
\]
For fixed $r$, we seek the values of $R$ for which $\A(r,1)$ admits
a homeomorphism onto $\A(R,1)$ whose coordinate functions are harmonic.
No symmetry is assumed for the map.

In dimension two, Nitsche formulated the corresponding conjecture in 1962
\cite{Nitsche}; the problem also arose from the geometry of doubly connected
minimal surfaces \cite{Nitsche1964,IKOMinimalSurfaces}.  A sequence of partial
estimates and structural results was obtained by Lyzzaik, Weitsman, and Kalaj
\cite{Lyzzaik,Weitsman,Kalaj2005}.  Iwaniec, Kovalev, and Onninen then
developed a sharp analysis of harmonic mappings of annuli \cite{IKO2010} and
proved the planar Nitsche conjecture in full \cite{IKO}.  The harmonic mapping problem for general doubly connected domains is
studied in \cite{IKOHarmonicProblem,BshoutyLyzzaikRasilaVasudevarao}.

Kalaj studied the three-dimensional problem \cite{Kalaj2005} and later
proposed the higher-dimensional bound
\[
 R\le\frac{nr}{n-1+r^n}
\]
for $n\ge3$ \cite{Kalaj}.

The candidate bounds are explained by radial harmonic maps of the form
\[
 h(t,\omega)=H(t)Q\omega,
 \qquad H(t)=At+Bt^{1-n},
 \qquad Q\in O(n).
\]
For an end-preserving critical map, the conditions $H(1)=1$ and $H'(r)=0$
give
\[
 H_{n,r}(t)=\frac{(n-1)t+r^nt^{1-n}}{n-1+r^n},
 \qquad
 H_{n,r}(r)=\frac{nr}{n-1+r^n}.
\]
If the ends are interchanged, the corresponding radial critical map is
selected by $H(r)=1$ and $H'(1)=0$ and has outer value
\[
 H_{n,r}^{-}(1)=\frac{nr^{n-1}}{1+(n-1)r^n}.
\]
The radial calculation suggests a sharp bound for each order of the
ends.  The problem is to prove these bounds for all harmonic
homeomorphisms, without a symmetry assumption, and to determine whether
equality forces the corresponding radial map up to an orthogonal
transformation.

Related Nitsche-type questions have also been
studied for annuli on
Riemann surfaces and for hyperbolic targets
\cite{Kalaj2011,KalajHyperbolic,Kalaj2016}.
Variational counterparts arise for $p$-harmonic, $n$-harmonic, weighted,
and nonlinear-elastic deformations
\cite{IwaniecOnninenP,AstalaIwaniecMartin,IwaniecOnninenNeo,
IwaniecKohKovalevOnninen,IwaniecOnninenN,MartinMcKubreJordens,
KoskiOnninen,Kalaj2019}.
These works use non-Euclidean target metrics or different energies.
Their Euler--Lagrange equations differ from the coordinatewise Laplace
equation considered here.

\subsection{Harmonic homeomorphisms}
The Nitsche problem is a special case of the harmonic mapping problem:
which pairs of domains can be joined by a harmonic homeomorphism
\cite{IKOHarmonicProblem}?  Harmonic maps are basic objects in geometric
analysis \cite{EellsSampson,SchoenUhlenbeck}.  For a Euclidean target,
their equation is
\[
 \Delta h=0
\]
coordinatewise.  Although the equation is linear, the class of
homeomorphisms is not convex.  Solving the Dirichlet problem alone
does not guarantee global injectivity.

In the plane, the Rad\'o--Kneser--Choquet theorem concerns homeomorphisms
from the unit circle onto the boundary of a convex domain.  Their
harmonic extensions are injective.  Lewy's theorem states that a planar
harmonic homeomorphism has a non-vanishing Jacobian.
Complex derivatives also play a central role in the theory of univalent
harmonic mappings \cite{Lewy,ClunieSheilSmall,Duren}.
In higher dimensions, a harmonic homeomorphism can have a vanishing
Jacobian \cite{Wood}.

There are special non-vanishing results for gradient maps.  In dimension
three, Lewy proved that if $u$ is harmonic and $Du$ is a homeomorphism,
then $\det D^2u$ does not vanish, so $Du$ is a diffeomorphism
\cite{LewyHarmonicGradients}.  Li, Liu, and Ma obtained a local
$p$-harmonic analogue under the assumption $|Du|>0$
\cite{LiLiuMaLocalLewy}.  These results require the components of the
map to come from a common potential.  A general vector-valued harmonic
homeomorphism need not have this form.

Boundary injectivity does not ensure injectivity of the harmonic
extension in higher dimensions.
Laugesen constructed sphere homeomorphisms whose coordinatewise harmonic
extensions to the ball are not injective \cite{Laugesen}.

These facts illustrate the global restrictions imposed by injectivity.
For concentric spherical annuli, the results below determine the sharp
restrictions on the two radius ratios and classify the maps attaining
them.  No continuous extension to the closed annulus is assumed.

\subsection{Main results}
Throughout the paper, harmonicity means that each coordinate function is
harmonic in the Euclidean sense.

Define
\begin{equation}\label{eq:Rpm}
 \Rplus(r)=\frac{nr}{n-1+r^n},
 \qquad
 \Rminus(r)=\frac{nr^{n-1}}{1+(n-1)r^n}.
\end{equation}
The next two theorems give the sharp bounds and classify the equality cases.

\begin{theorem}[Sharp bounds]\label{thm:sharp-bounds}
Let $n\ge3$, $0<r,R<1$, and let
\[
 h:\A(r,1)\overset{\rm onto}{\longrightarrow}\A(R,1)
\]
be a homeomorphism with harmonic coordinate functions.

\begin{enumerate}[label=\rm(\roman*)]
\item One always has
\[
 R\le \Rplus(r)=\frac{nr}{n-1+r^n}.
\]
\item If the inner end of the source corresponds to the outer end of the target, then
\[
 R\le \Rminus(r)=\frac{nr^{n-1}}{1+(n-1)r^n}.
\]
This bound is strictly stronger than the one in \rm(i).
\end{enumerate}
Both constants are sharp.
\end{theorem}

\begin{theorem}[Critical rigidity]\label{thm:critical-rigidity}
Under the hypotheses of Theorem~\ref{thm:sharp-bounds}, equality is
classified as follows.
\begin{enumerate}[label=\rm(\roman*)]
\item If $R=\Rplus(r)$, then there exists $Q\in O(n)$ such that, for
$r<t<1$ and $\omega\in\Sph$,
\begin{equation}\label{eq:critical-map}
 h(t,\omega)=H_{n,r}(t)Q\omega,
 \qquad
 H_{n,r}(t)=\frac{(n-1)t+r^nt^{1-n}}{n-1+r^n}.
\end{equation}
\item If the ends are interchanged and $R=\Rminus(r)$, then there exists
$Q\in O(n)$ such that, for $r<t<1$ and $\omega\in\Sph$,
\begin{equation}\label{eq:reverse-critical-map}
 h(t,\omega)=H_{n,r}^{-}(t)Q\omega,
 \qquad
 H_{n,r}^{-}(t)=
 \frac{r^{n-1}\bigl((n-1)t+t^{1-n}\bigr)}{1+(n-1)r^n}.
\end{equation}
\end{enumerate}
\end{theorem}

By scaling we immediately obtain the following form.

\begin{corollary}\label{cor:unscaled}
Let $n\ge3$, $0<r_0<r_1$ and $0<R_0<R_1$. If there is a harmonic homeomorphism
$h:\A(r_0,r_1)\to\A(R_0,R_1)$, then
\[
 \frac{R_0}{R_1}\le
 \frac{n(r_0/r_1)}{n-1+(r_0/r_1)^n}.
\]
If equality holds, then for some $Q\in O(n)$,
\[
 h(x)=R_1H_{n,r_0/r_1}(|x|/r_1)Q\frac{x}{|x|}.
\]
If $h$ interchanges the ends, then
\[
 \frac{R_0}{R_1}\le
 \frac{n(r_0/r_1)^{n-1}}{1+(n-1)(r_0/r_1)^n}.
\]
Equality in this reversed bound holds only for
\[
 h(x)=R_1H_{n,r_0/r_1}^{-}(|x|/r_1)Q\frac{x}{|x|}
\]
with $Q\in O(n)$.
\end{corollary}

The proofs use the following theorem.

\begin{theorem}[Sharp spherical coupling contraction]
\label{thm:coupling-contraction}
Let $n\ge2$, let $\mu$ be normalized surface measure on $\Sph$, and let
$K:[-1,1]\to[0,\infty)$ be bounded, Borel, and nondecreasing.
Its degree-one multiplier $\kappa_1$ is characterized by
\[
 \int_{\Sph}K(z\cdot\omega)\omega\dd\mu(\omega)=\kappa_1z,
 \qquad z\in\Sph.
\]
For every Borel probability measure $\gamma$ on $\Sph_x\times\Sph_\omega$
whose first marginal is $\mu$, define
\[
 \Ncal_\gamma(E)=\int_{\Sph\times E}x\dd\gamma(x,\omega),
 \qquad
 (\Kcal\Ncal_\gamma)(z)
 =\int_{\Sph\times\Sph}K(z\cdot\omega)x\dd\gamma(x,\omega).
\]
Then
\begin{equation}\label{eq:coupling-contraction}
 \sup_{\substack{\gamma\in\mathcal P(\Sph\times\Sph)\\
                  (\pi_1)_\#\gamma=\mu}}
 \|\Kcal\Ncal_\gamma\|_{L^1(\mu)}=\kappa_1,
\end{equation}
where $\mathcal P$ denotes Borel probability measures and
$\pi_1(x,\omega)=x$.  If $K$ is also strictly
increasing, then equality
$\|\Kcal\Ncal_\gamma\|_{L^1(\mu)}=\kappa_1$ holds if and only if
\begin{equation}\label{eq:orthogonal-coupling}
 \gamma=(\operatorname{id},M)_\#\mu
 \qquad\text{for some }M\in O(n).
\end{equation}
\end{theorem}

Only the first marginal is prescribed.  The optimal constant is the
degree-one multiplier of $K$.  Both the bound and its equality classification
are proved in Section~\ref{sec:coupling} by equidistributed scalar
projections and convex duality.

\begin{remark}\normalfont
In a subsequent preprint, Zhong and Kalaj
\cite[Theorems~1.2 and~1.3]{ZhongKalajMeanWidth}
proved a sharp mean-width contraction theorem for continuous
monotone zonal kernels, with rigidity for strictly monotone kernels.
These statements are equivalent, by support-function duality,
to the \emph{continuous-kernel} case of
Theorem~\ref{thm:coupling-contraction}, including its equality classification.
\end{remark}

\subsection{Consequences and extensions}\label{sec:consequences}
For $n\ge3$,
\[
 R_{n,-}(r)=\frac{nr^{n-1}}{1+(n-1)r^n}
 <
 \frac{nr}{n-1+r^n} = R_{n,+}(r).
\]
Thus interchanging the ends gives a stricter bound.  In dimension two
the constants coincide, but the two radial equality maps are different.

The bounds and their equality classifications extend to bounded harmonic maps of nonzero degree, as follows.

\begin{theorem}[Nonzero-degree maps]\label{thm:nonzero-degree}
Let $n\ge2$, $0<r,R<1$, and let
$h:\A(r,1)\to\mathbb R^n$ be bounded and coordinatewise harmonic.
Assume that $|h|\ge R$ throughout $\A(r,1)$ and that, for one
$t_0\in(r,1)$, the direction map
\[
 q_{t_0}(\omega)=\frac{h(t_0,\omega)}{|h(t_0,\omega)|}
\]
has nonzero degree.
\begin{enumerate}[label=\rm(\roman*)]
\item If
\[
 \begin{aligned}
 \bigl\||h(t,\cdot)|-R\bigr\|_{L^2(\mu)}&\longrightarrow0
 &&(t\downarrow r),\\
 \bigl\||h(t,\cdot)|-1\bigr\|_{L^2(\mu)}&\longrightarrow0
 &&(t\uparrow1),
 \end{aligned}
\]
then $R\le\Rplus(r)$.  Equality holds if and only if
\[
 h(t,\omega)=H_{n,r}(t)Q\omega
\]
for some $Q\in O(n)$, with $H_{n,r}$ as in \eqref{eq:critical-map}.
\item If the two endpoint modulus limits are instead $1$ and $R$,
respectively, then $R\le\Rminus(r)$.  Equality holds if and only if
\[
 h(t,\omega)=H_{n,r}^{-}(t)Q\omega
\]
for some $Q\in O(n)$, with $H_{n,r}^{-}$ as in
\eqref{eq:reverse-critical-map}.
\end{enumerate}
Both bounds are sharp.  In particular, if the common degree of the
interior direction maps has absolute value at least two, the inequality
in the corresponding case is strict.
\end{theorem}

The proof is given in Section~\ref{sec:nonzero-degree}.
In the plane, Corollary~\ref{cor:IKO-winding} proves
\cite[Conjecture~1.8]{IKO} for maps of nonzero winding number without
endpoint modulus conditions.
Corollary~\ref{cor:constraint-radius} gives a sharp lower bound for the
contact radius of a spherical constraint map.  It also determines the
exterior harmonic map at equality.

\subsection{Proof strategy}
The proof combines the annular Poisson representation with a
selector--coupling mechanism.  The Poisson operators $A_t$ and $B_t$
extend data on the inner and outer boundary, respectively, with zero
data on the other boundary.  They have positive zonal kernels that
are strictly increasing in the scalar product; see
Lemma~\ref{lem:kernel-monotone}.  With measurable sphere-valued
normalized traces $\phi,\psi$, a harmonic homeomorphism has the
representation
\[
 h(t,\cdot)=
 \begin{cases}
 R A_t\phi+B_t\psi,&\text{if the ends are preserved},\\
 A_t\phi+R B_t\psi,&\text{if the ends are interchanged}.
 \end{cases}
\]
Only $L^2$ boundary traces are required.

On each interior sphere, the direction map
$q_t(\omega)=h(t,\omega)/|h(t,\omega)|$ has nonzero degree and is onto.
Choose a Borel \emph{selector}, that is, a right inverse $\sigma_t$
of $q_t$, and set
\[
 \gamma_t=(\operatorname{id},\sigma_t)_\#\mu,
 \qquad m_t=(\sigma_t)_\#\mu,
 \qquad \Ncal_t(E)=\int_{\Sph\times E}x\dd\gamma_t(x,\omega).
\]
Here $\mu$ denotes normalized surface measure on $\Sph$, so that $\mu(\Sph)=1$.
The first coordinate $x$ is an image direction and the second coordinate
$\omega$ is a selected source direction.  Thus $\gamma_t$ has first
marginal $\mu$, second marginal $m_t$, and satisfies
$x=q_t(\omega)$ almost everywhere.  This alignment turns the modulus
into a linear pairing:
\[
 \int |h(t,\omega)|\dd m_t(\omega)
 =\int x\cdot h(t,\omega)\dd\gamma_t(x,\omega)>R.
\]
Writing $\langle f,g\rangle=\int f\cdot g\dd\mu$, symmetry and
Fubini move a kernel operator from the boundary data to $\Ncal_t$;
the detailed justification is given in \eqref{eq:Fubini-swap}.
Theorem~\ref{thm:coupling-contraction} then gives the core estimate
\[
 \begin{aligned}
 R&<R\langle\phi,A_t\Ncal_t\rangle
       +\langle\psi,B_t\Ncal_t\rangle\\
  &\le R\|A_t\Ncal_t\|_1+\|B_t\Ncal_t\|_1
   \le R a_1(t)+b_1(t)
 \end{aligned}
\]
in the end-preserving case, and $R<a_1(t)+Rb_1(t)$ in the
end-reversing case.  Here $a_1,b_1$ are the degree-one Poisson
multipliers in \eqref{eq:first-multipliers}.  One-sided differentiation
at $t=r$ and $t=1$, respectively, gives the sharp bounds.

The coupling theorem itself follows from a scalar observation.
For an independent uniform test direction $z$ and a Borel unit vector
field $e(z)$ that realizes the norm, the projections
$e(z)\cdot x$ and $z\cdot\omega$ have the same distribution.
The convex potential $\Phi(t)=\int_0^tK(s)\dd s$ therefore bounds
their correlation with $K$ by the degree-one multiplier.  Under strict
increase, the convexity deficit vanishes only when the two projections
agree.  Compatibility over the test directions then forces an
orthogonal graph.  Section~\ref{sec:coupling} gives this proof.

Critical rigidity requires additional endpoint analysis.  The
couplings with first marginal $\mu$ form a weakly compact class,
although limits of graph couplings need not remain graph couplings.
At a critical radius the finite-slice defect is of second order,
whereas the Poisson operator from the opposite boundary vanishes
to first order.  Dividing by the distance to the endpoint and using
the smooth transfer-kernel limit forces equality in the coupling
theorem.  The resulting orthogonal graph identifies one boundary trace.

To recover the other trace, the modulus constraint gives a local
variance estimate at the source sphere carrying the trace of modulus $R$.
It yields a $C^{0,1/2}$ trace and uniform convergence of $q_t$ there.
The direction relation thus passes to the orthogonal limiting coupling
and determines the same $Q\in O(n)$ for both traces. The Poisson
representation then determines the whole map.

The planar proof of Iwaniec--Kovalev--Onninen uses \emph{quadratic}
circular means, complex derivatives and Jacobian identities, and
Fourier--Laurent expansions \cite{IKO2010,IKO}.  Here the selector
produces the linear pairing $|h|=x\cdot h$, and the coupling theorem
controls it uniformly over all selectors.  The argument also proves
the planar Nitsche theorem by applying the endpoint argument in both cases,
as shown in the proof of Theorem~\ref{thm:nonzero-degree}. The separation
of surjectivity, kernel contraction, and endpoint analysis may also be
useful for other integral representations with monotone zonal kernels.

Section~\ref{sec:boundary-reduction} treats radial maps and boundary
traces; Sections~\ref{sec:coupling}--\ref{sec:sharp-bounds} establish
the contraction and the sharp bounds.  Section~\ref{sec:critical-rigidity}
proves critical rigidity, and Section~\ref{sec:nonzero-degree} proves Theorem~\ref{thm:nonzero-degree} and gives the planar winding-number bound and the spherical constraint-map corollary.

\section{Extremal radial maps and boundary reduction}
\label{sec:boundary-reduction}

The boundary and kernel arguments below are stated for $n\ge2$.
The logarithmic zero mode is shown explicitly when $n=2$.

\subsection{Extremal radial maps and sharpness}\label{sec:radial}
Consider a radial map
\[
 h(t,\omega)=H(t)Q\omega,\qquad Q\in O(n).
\]
The harmonicity of the coordinate functions is equivalent to
\[
 H''+\frac{n-1}{t}H'-\frac{n-1}{t^2}H=0,
\]
whose general solution is $H(t)=At+Bt^{1-n}$.

For an end-preserving critical map, impose $H(1)=1$ and $H'(r)=0$. This gives
\begin{equation}\label{eq:Hplus}
 H_+(t)=\frac{(n-1)t+r^nt^{1-n}}{n-1+r^n},
 \qquad
 H_+(r)=\Rplus(r),
\end{equation}
and
\[
 H_+'(t)=\frac{n-1}{n-1+r^n}\left(1-\frac{r^n}{t^n}\right)>0
 \quad(r<t<1).
\]
Thus $H_+(t)Q\omega$ is a harmonic homeomorphism from $\A(r,1)$ onto
$\A(\Rplus(r),1)$.

For an end-reversing critical map, impose $H(r)=1$ and $H'(1)=0$. This yields
\begin{equation}\label{eq:Hminus}
 H_-(t)=\frac{r^{n-1}\bigl((n-1)t+t^{1-n}\bigr)}{1+(n-1)r^n},
 \qquad
 H_-(1)=\Rminus(r),
\end{equation}
and
\[
 H_-'(t)=\frac{(n-1)r^{n-1}}{1+(n-1)r^n}(1-t^{-n})<0.
\]
Hence the constants in Theorem~\ref{thm:sharp-bounds} are sharp once the upper
bounds are established.

\subsection{Ends and measurable boundary traces}\label{sec:boundary-traces}
We next isolate the topological and boundary-regularity facts needed to
replace continuous boundary values by measurable traces.  No extension of
$h$ to the closed annulus is assumed.

\begin{lemma}\label{lem:ends}
Let $n\ge2$.  Every homeomorphism $h:\A(r,1)\to\A(R,1)$ induces a permutation of the two ends. If it preserves the ends, then
\[
 \sup_{\omega\in\Sph}\bigl||h(t,\omega)|-R\bigr|\to0\quad(t\downarrow r),
 \qquad
 \sup_{\omega\in\Sph}\bigl||h(t,\omega)|-1\bigr|\to0\quad(t\uparrow1).
\]
If it reverses the ends, the two limits are interchanged.
\end{lemma}

\begin{proof}
A homeomorphism is a proper homotopy equivalence and hence extends to
a homeomorphism of the Freudenthal compactifications, inducing a bijection
of the ends \cite[Proposition~3.3.12(a), (c)]{GuilbaultEnds}.
Full radial collars form a neighborhood basis at each annular end.
Thus every collar of the corresponding target end contains the image
of a sufficiently thin source collar, giving the stated uniform limits.
\end{proof}

Let $\mu$ denote normalized surface measure on $\Sph$.

\begin{lemma}\label{lem:traces}
Assume $n\ge2$, and let $u$ be bounded and harmonic on $\A(r,1)$.  Then
there are unique $g_r,g_1\in L^\infty(\Sph,\mu)$ such that
\[
 u(t,\cdot)\to g_r\quad\text{in }L^2\ (t\downarrow r),
 \qquad
 u(t,\cdot)\to g_1\quad\text{in }L^2\ (t\uparrow1).
\]
They satisfy
\[
 \|g_r\|_\infty,\|g_1\|_\infty\le\|u\|_\infty.
\]
The function $u$ is the unique bounded harmonic function with these
$L^2$ boundary traces.  We call it the annular Poisson extension and write
\[
 u=P_{\A(r,1)}[g_r,g_1].
\]
Define the two boundary propagation operators by
\[
 A_tg=P_{\A(r,1)}[g,0](t,\cdot),
 \qquad
 B_tg=P_{\A(r,1)}[0,g](t,\cdot).
\]
Their boundary traces are explicitly
\[
 \begin{array}{c|cc}
  &t\downarrow r&t\uparrow1\\ \hline
  A_tg&g&0\\
  B_tg&0&g
 \end{array}
 \qquad\text{in }L^2(\Sph),
\]
and therefore
\[
 u(t,\cdot)=A_tg_r+B_tg_1,
\]
so $A_t$ carries the inner trace and $B_t$ carries the outer trace.

For a harmonic homeomorphism $h:\A(r,1)\to\A(R,1)$ in any
dimension $n\ge2$, there exist measurable
$\phi,\psi\in L^\infty(\Sph;\Sph)$ such that
\begin{align}
 h(t,\cdot)&=R A_t\phi+B_t\psi &&\text{in the end-preserving case},\label{eq:Poisson-pres}\\
 h(t,\cdot)&=A_t\phi+R B_t\psi &&\text{in the end-reversing case}.\label{eq:Poisson-rev}
\end{align}
\end{lemma}

\begin{proof}
The bounded Dirichlet theorem for Lipschitz domains
\cite[Proposition, p.~64]{SchmidtHarmonic} identifies bounded harmonic
functions with their almost-everywhere nontangential $L^\infty$ traces,
with the Poisson integral as the inverse. On either sphere of the
annulus the radial approach is nontangential. Pulling the traces back to
$\Sph$ and using dominated convergence gives the two strong $L^2$
limits. The trace norm bounds and uniqueness follow from the same
theorem. Applying it to the data $(g,0)$ and $(0,g)$ gives the stated
limits for $A_t$ and $B_t$, and linearity gives the decomposition of $u$.

To compute the classical annular Poisson multipliers
\cite[Chapter~10, (10.9)--(10.12) and Exercise~10]{AxlerBourdonRamey},
let $u_{\ell,m}(t)$ be the coefficient of $u(t,\cdot)$ in an orthonormal
basis of spherical harmonics, with $\ell$ denoting the degree.
Harmonicity gives
\begin{equation}\label{eq:mode-ode}
 u_{\ell,m}''+\frac{n-1}{t}u_{\ell,m}'-
 \frac{\ell(\ell+n-2)}{t^2}u_{\ell,m}=0.
\end{equation}
The strong $L^2$ limits identify the endpoint values of $u_{\ell,m}$
as the trace coefficients $g_{r,\ell,m}$ and $g_{1,\ell,m}$.
Let $a_\ell$ and $b_\ell$ solve \eqref{eq:mode-ode} with boundary values
\[
 \begin{array}{c|cc}
  &t=r&t=1\\ \hline
  a_\ell&1&0\\
  b_\ell&0&1.
 \end{array}
\]
By uniqueness for this two-point problem,
$u_{\ell,m}(t)=a_\ell(t)g_{r,\ell,m}+b_\ell(t)g_{1,\ell,m}$.
For $n\ge3$, and also for $n=2$ and $\ell\ge1$, solving the problem gives
\begin{align}
 a_\ell(t)&=
 \frac{r^{\ell+n-2}\bigl(t^{-(\ell+n-2)}-t^\ell\bigr)}
 {1-r^{2\ell+n-2}},\label{eq:aell}\\
 b_\ell(t)&=
 \frac{t^\ell-r^{2\ell+n-2}t^{-(\ell+n-2)}}
 {1-r^{2\ell+n-2}}.\label{eq:bell}
\end{align}

For the zero mode, we have
\begin{equation}\label{eq:zero-multipliers}
 a_0(t)=
 \begin{cases}
 \displaystyle\frac{r^{n-2}(t^{2-n}-1)}{1-r^{n-2}},&n\ge3,\\[6pt]
 \displaystyle\frac{-\log t}{\log(1/r)},&n=2,
 \end{cases}
 \qquad b_0(t)=1-a_0(t).
\end{equation}
For every $\ell$, the two multipliers lie in $[0,1]$ on $[r,1]$.

Apply the trace assertion coordinatewise to $h$. By Lemma~\ref{lem:ends}
and the inequality $\bigl||v|-|w|\bigr|\le |v-w|$, the vector traces
satisfy $|g_r|=R$ and $|g_1|=1$ almost everywhere in the end-preserving
case. Taking $\phi=g_r/R$ and $\psi=g_1$ gives \eqref{eq:Poisson-pres}.
In the end-reversing case, $|g_r|=1$ and $|g_1|=R$ almost everywhere,
so $\phi=g_r$ and $\psi=g_1/R$ give \eqref{eq:Poisson-rev}.
\end{proof}

For later use, recall that $a_1(t)$ and $b_1(t)$ are the $\ell=1$
cases of the multipliers in \eqref{eq:aell}--\eqref{eq:bell}.  Thus
$A_tY=a_1(t)Y$ and $B_tY=b_1(t)Y$ for every spherical harmonic $Y$
of degree one, and their explicit formulas are
\begin{equation}\label{eq:first-multipliers}
 a_1(t)=\frac{r^{n-1}(t^{1-n}-t)}{1-r^n},
 \qquad
 b_1(t)=\frac{t-r^nt^{1-n}}{1-r^n}.
\end{equation}
Their endpoint derivatives are
\begin{align}
 a_1'(r)&=-\frac{n-1+r^n}{r(1-r^n)}, &
 b_1'(r)&=\frac{n}{1-r^n},\label{eq:inner-derivatives}\\
 a_1'(1)&=-\frac{nr^{n-1}}{1-r^n}, &
 b_1'(1)&=\frac{1+(n-1)r^n}{1-r^n}.\label{eq:outer-derivatives}
\end{align}

\section{The spherical coupling contraction}\label{sec:coupling}
The proof of the sharp bounds requires a contraction estimate for vector
measures.  After verifying monotonicity of the annular Poisson kernels,
we prove the full coupling theorem by equidistributed scalar
projections and convex duality.

\subsection{Annular Poisson kernels}\label{sec:poisson}
For $r<t<1$, both $A_t$ and $B_t$ have positive zonal kernels.  Their
strict angular monotonicity supplies the convexity and strictness
used in Theorem~\ref{thm:coupling-contraction}.

\begin{lemma}\label{lem:kernel-monotone}
Let $n\ge2$.  For every fixed $r<t<1$, there are positive functions $K_A(t,\cdot)$ and $K_B(t,\cdot)$ on $[-1,1]$ such that
\[
 A_tf(\omega)=\int_{\Sph}K_A(t,\omega\cdot\xi)f(\xi)\dd\mu(\xi),
\]
with the analogous formula for $B_t$, and both kernels are strictly increasing in the scalar product.
\end{lemma}

\begin{proof}
We prove first that each boundary operator has a positive zonal kernel and
then compare two boundary points by reflection.  Let $D=\A(r,1)$, and let
$G_D(y,z)$ be the positive Dirichlet Green function normalized by
$-\Delta_zG_D(y,z)=\delta_y$.  Its Poisson density with respect to ordinary
surface measure is
\[
 p_D(y,\zeta)=-\partial_{\nu_D(\zeta)}G_D(y,\zeta)>0,
\]
by the Hopf boundary lemma.
We use here only the standard Green-function, maximum-principle, and local
Hopf theory on smooth bounded domains; see \cite[Chapters~2 and~6]{GT}.

Since $\dd\mu=\dd\sigma/|\Sph|$, the kernels in the statement are
\begin{equation}\label{eq:KA-KB-Poisson}
 \begin{aligned}
 K_A(t,\omega\cdot\xi)
  &=|\Sph|r^{n-1}p_D(t\omega,r\xi),\\
 K_B(t,\omega\cdot\xi)
  &=|\Sph|p_D(t\omega,\xi).
 \end{aligned}
\end{equation}
Rotational invariance makes them zonal, and the displayed positivity of
$p_D$ makes them positive.

Fix $y=t\omega$ and suppose
$\omega\cdot\xi_1>\omega\cdot\xi_2$.  Let $S$ be reflection across
\[
 \Pi=\{z:z\cdot(\xi_1-\xi_2)=0\},
 \qquad
 H=\{z:z\cdot(\xi_1-\xi_2)>0\}.
\]
Then $S\xi_1=\xi_2$, while $y,b\xi_1\in H$ and
$Sy\notin\overline H$ for $b\in\{r,1\}$.  Moreover, $S(D)=D$.  In
$D\cap H$ define
\[
 W(z)=G_D(y,z)-G_D(Sy,z).
\]
On the planar boundary $z\in\Pi$, reflection invariance and $Sz=z$ give
\[
 G_D(Sy,z)=G_D(y,Sz)=G_D(y,z),
\]
whereas on the spherical boundary both Green functions vanish by the
Dirichlet condition.  Thus $W=0$ on all of $\partial(D\cap H)$.  On
$(D\cap H)\setminus\overline B_\varepsilon(y)$, the function $W$ is
harmonic, and for small $\varepsilon$ its Green-function singularity makes
$W>0$ on $\partial B_\varepsilon(y)$.  In dimension two the singularity
is $(2\pi)^{-1}\log(1/|z-y|)$, which also tends to $+\infty$.
The weak and strong maximum principles therefore give
\[
 W>0\qquad\text{in }(D\cap H)\setminus\{y\}.
\]
The point $b\xi_1$ is a smooth boundary point away from the reflecting
plane and the pole.  The Hopf lemma yields
\[
 -\partial_{\nu_D}W(b\xi_1)>0.
\]
Using the reflection identity for $G_D(Sy,\cdot)$ gives
\[
 p_D(y,b\xi_1)-p_D(y,b\xi_2)>0.
\]
Reflection preserves the physical outward normal on either spherical
boundary component, including its sign at $b=r$.  Applying
\eqref{eq:KA-KB-Poisson} with $b=r$ and $b=1$ proves the strict increase of
$K_A$ and $K_B$, respectively.
\end{proof}

\subsection{Equidistributed projections and convex duality}
Only the first marginal of the coupling is prescribed.  Rotational
invariance nevertheless fixes the distributions of two scalar
projections, and a one-dimensional convexity inequality gives both the
sharp bound and its equality classification.

\begin{proof}[Proof of Theorem~\ref{thm:coupling-contraction}]
Fix a unit vector $v_0\in\Sph$, and let
\[
 \rho_n=(x\mapsto v_0\cdot x)_\#\mu,\qquad
 \dd\rho_n(t)=
 \frac{\Gamma(n/2)}{\sqrt{\pi}\,\Gamma((n-1)/2)}
 (1-t^2)^{(n-3)/2}\dd t,\quad -1<t<1.
\]
This probability measure is independent of $v_0$ and is symmetric about
zero.  Taking $z=v_0$ in the identity defining $\kappa_1$ and then
taking the scalar product with $v_0$, we obtain
\[
 \begin{aligned}
 \kappa_1
 &=\int_{\Sph}(v_0\cdot x)K(v_0\cdot x)\dd\mu(x)\\
 &=\int_{-1}^1 tK(t)\dd\rho_n(t)\\
 &=\int_{(0,1]}t\bigl(K(t)-K(-t)\bigr)\dd\rho_n(t)\ge0.
 \end{aligned}
\]
The second equality is the push-forward integration formula for
$\rho_n$, applied to $t\mapsto tK(t)$; the third uses symmetry under
$t\mapsto-t$, and the last inequality follows from monotonicity of $K$.
Define the bounded Borel vector field
\[
 F(z)=(\Kcal\Ncal_\gamma)(z)
 {}=\int_{\Sph\times\Sph}K(z\cdot\omega)x\dd\gamma(x,\omega)
\]
and choose the Borel unit vector field
\[
 e(z)=
 \begin{cases}
 F(z)/|F(z)|,&F(z)\ne0,\\
 v_0,&F(z)=0.
 \end{cases}
\]
Thus $e(z)\cdot F(z)=|F(z)|$ everywhere; in particular, $e$ is a unit
vector even on the zero set of $F$.

On $\Sph_z\times\Sph_x\times\Sph_\omega$, define the product
probability measure $\rho$ and the Borel functions $U,V$ by
\[
 \rho=\mu\otimes\gamma,\qquad
 U(z,x,\omega)=e(z)\cdot x,\qquad
 V(z,x,\omega)=z\cdot\omega.
\]
To identify the two push-forward measures, let
$f:[-1,1]\to\mathbb R$ be an arbitrary bounded Borel function.
All integrands below are bounded and Borel, so Fubini's theorem applies.
For $U$, we first use the $x$-marginal of $\gamma$, and then rotational
invariance with the fixed unit vector $e(z)$:
\[
 \begin{aligned}
 \int f(U)\dd\rho
 &=\int_{\Sph}\int_{\Sph\times\Sph}
 f(e(z)\cdot x)\dd\gamma(x,\omega)\dd\mu(z)\\
 &=\int_{\Sph}\int_{\Sph}
 f(e(z)\cdot x)\dd\mu(x)\dd\mu(z)\\
 &=\int_{-1}^1 f(t)\dd\rho_n(t).
 \end{aligned}
\]
For $V$, we instead integrate first in $z$.  For every fixed
$(x,\omega)$, the unit vector $\omega$ determines the same projection
measure $\rho_n$, so
\[
 \begin{aligned}
 \int f(V)\dd\rho
 &=\int_{\Sph\times\Sph}\int_{\Sph}
 f(z\cdot\omega)\dd\mu(z)\dd\gamma(x,\omega)\\
 &=\int_{\Sph\times\Sph}
 \left(\int_{-1}^1 f(t)\dd\rho_n(t)\right)\dd\gamma(x,\omega)\\
 &=\int_{-1}^1 f(t)\dd\rho_n(t).
 \end{aligned}
\]
Since these identities hold for every bounded Borel $f$, the definition
of a push-forward measure gives
\begin{equation}\label{eq:projection-laws}
 U_\#\rho=V_\#\rho=\rho_n.
\end{equation}
No condition on the second marginal of $\gamma$, or on $e_\#\mu$, is
used here.

The choice of $e$ gives $e(z)\cdot F(z)=|F(z)|$ pointwise, including
where $F(z)=0$.  Hence the $L^1(\mu)$ norm has the representation
\[
 \begin{aligned}
 \|\Kcal\Ncal_\gamma\|_1
 &=\int_{\Sph}|F(z)|\dd\mu(z)
  =\int_{\Sph}e(z)\cdot F(z)\dd\mu(z)\\
 &=\int_{\Sph}\int_{\Sph\times\Sph}
 (e(z)\cdot x)K(z\cdot\omega)\dd\gamma(x,\omega)\dd\mu(z)\\
 &=\int UK(V)\dd\rho.
 \end{aligned}
\]
Here the factors are $U=e(z)\cdot x$ and $K(V)=K(z\cdot\omega)$.
Applying \eqref{eq:projection-laws} to $t\mapsto tK(t)$ also gives
\[
 \int UK(U)\dd\rho=\int VK(V)\dd\rho
 =\int_{-1}^1 tK(t)\dd\rho_n(t)=\kappa_1.
\]
Thus the norm involves the mixed integral $\int UK(V)\dd\rho$, while
$\kappa_1$ is determined by the common push-forward measure $\rho_n$.

Set
\[
 \Phi(t)=\int_0^tK(r)\dd r,\qquad
 D_\Phi(s,t)=\Phi(s)-\Phi(t)-K(t)(s-t).
\]
Monotonicity of $K$ makes $\Phi$ convex.  More explicitly, using the
oriented Lebesgue integral,
\[
 D_\Phi(s,t)=\int_t^s\bigl(K(r)-K(t)\bigr)\dd r\ge0,
 \qquad s,t\in[-1,1].
\]
This inequality holds also at jump points and at the endpoints; no
continuity or differentiability of $K$ is required.
Since $\Phi$ is bounded and continuous,
\eqref{eq:projection-laws} gives
\[
 \int\bigl(\Phi(U)-\Phi(V)\bigr)\dd\rho=0.
\]
Using the preceding representations of $\kappa_1$ and the norm, and
adding this zero integral, we obtain
\begin{equation}\label{eq:convex-defect}
 \begin{aligned}
 \kappa_1-\|\Kcal\Ncal_\gamma\|_1
 &=\int\bigl[VK(V)-UK(V)\bigr]\dd\rho
 =\int D_\Phi(U,V)\dd\rho\\
 &=\int_{\Sph}\int_{\Sph\times\Sph}
 D_\Phi(e(z)\cdot x,z\cdot\omega)
 \dd\gamma(x,\omega)\dd\mu(z)\ge0.
 \end{aligned}
\end{equation}
The same identity gives the rigidity under strict increase of $K$.
Indeed, in that case the integral defining $D_\Phi(s,t)$ is strictly
positive for $s\ne t$, so
\[
 D_\Phi(s,t)=0\quad\Longleftrightarrow\quad s=t.
\]
Equality in \eqref{eq:convex-defect} therefore implies
\begin{equation}\label{eq:scalar-linearization}
 e(z)\cdot x=z\cdot\omega
 \quad\text{for }\rho\text{-almost every }(z,x,\omega).
\end{equation}
Define the real matrix
\[
 M=n\int_{\Sph}z\,e(z)^T\dd\mu(z).
\]
By Fubini, for $\gamma$-almost every $(x,\omega)$ we may multiply
\eqref{eq:scalar-linearization} by $z$ and integrate in $z$.
Since $\int zz^T\dd\mu(z)=I/n$, this gives
\[
 Mx=n\int_{\Sph}z(e(z)\cdot x)\dd\mu(z)
    =n\int_{\Sph}z(z\cdot\omega)\dd\mu(z)=\omega.
\]
The first marginal condition and $|\omega|=1$ imply $|Mx|=1$ for
$\mu$-almost every $x$.  Continuity and the full support of $\mu$
extend this identity to every $x\in\Sph$, whence $M^TM=I$.
Thus $M\in O(n)$; since $\gamma$ is concentrated on $\omega=Mx$
and its first marginal is $\mu$, it equals
$(\operatorname{id},M)_\#\mu$.

Finally, for every $M\in O(n)$, rotational invariance gives
\[
 (\Kcal\Ncal_{(\operatorname{id},M)_\#\mu})(z)
 =\int_{\Sph}K(z\cdot Mx)x\dd\mu(x)
 =\kappa_1M^Tz.
\]
Since $\kappa_1\ge0$, every such coupling attains the upper bound
in \eqref{eq:convex-defect}.  This proves the sharp bound
\eqref{eq:coupling-contraction} for every nondecreasing $K$ and the
converse equality assertion when $K$ is strictly increasing, completing
the proof.
\end{proof}

\begin{remark}
\label{rem:earlier-geometric-proof}
\normalfont
The first arXiv version of this paper
(\href{https://arxiv.org/abs/2608.27303v1}{arXiv:2608.27303v1},
Propositions~3.3 and~5.1) contains an alternative geometric proof based
on a layer-cake decomposition, spherical-cap barycenter extremality,
and strict concavity of the cap-barycenter profile.
The present proof removes the continuity assumption on $K$ from the
rigidity statement, which now holds for bounded nonnegative strictly
increasing Borel kernels.
\end{remark}

\section{Finite slices and the sharp bounds}\label{sec:sharp-bounds}
We now apply Theorem~\ref{thm:coupling-contraction} to the direction map on each
interior sphere.  The construction stays at a fixed $t\in(r,1)$ and
therefore does not require a boundary extension of $h$.
For $r<t<1$, define
\[
 q_t(\omega)=\frac{h(t,\omega)}{|h(t,\omega)|}.
\]

\begin{lemma}\label{lem:slice}
Let $n\ge2$, and let $h:\A(r,1)\to\A(R,1)$ be a harmonic
homeomorphism.  The map $q_t:\Sph\to\Sph$ has degree $\pm1$ and is onto. There exists a Borel right inverse
$\sigma_t:\Sph\to\Sph$ with $q_t(\sigma_t(x))=x$. If
\[
 \gamma_t=(x,\sigma_t(x))_\#\mu,
 \qquad
 m_t=(\sigma_t)_\#\mu,
\]
then the first marginal of $\gamma_t$ is $\mu$ and
\begin{equation}\label{eq:slice-identity}
 \int x\cdot h(t,\omega)\dd\gamma_t(x,\omega)
 =\int |h(t,\omega)|\dd m_t(\omega)>R.
\end{equation}
\end{lemma}

\begin{proof}
The inclusion $j_t(\omega)=t\omega$ and the radial projection $p(y)=y/|y|$ are homotopy equivalences. Since $h$ is a homeomorphism,
$q_t=p\circ h\circ j_t$ induces an isomorphism on $H_{n-1}(\Sph;\mathbb Z)$, hence has degree $\pm1$ and is surjective.
These are the standard homological properties of degree
\cite[Section~2.2]{Hatcher}.
The selection argument below only requires that $q_t$ have nonzero
degree; injectivity is not used once surjectivity is known.

Since $q_t$ is a continuous surjection between compact metric spaces, the
Borel right-inverse theorem gives a Borel map $\sigma_t$ satisfying
$q_t\circ\sigma_t=\operatorname{id}_{\Sph}$
\cite[Theorem~6.9.7]{Bogachev}.  On the support of $\gamma_t$
we have $x=q_t(\omega)$, so
$x\cdot h(t,\omega)=|h(t,\omega)|$.  The compact set
$h(\{t\}\times\Sph)$ lies inside the open target annulus, hence its modulus
has minimum strictly larger than $R$.  This proves
\eqref{eq:slice-identity}.
\end{proof}

For fixed $t\in(r,1)$, the smoothness of $q_t$ implies $m_t\ll\mu$:
if $E$ is a $\mu$-null Borel set, then $q_t(E)$ is $\mu$-null because
$q_t$ is Lipschitz, and $\sigma_t^{-1}(E)\subset q_t(E)$.

Let $\Ncal_t=\Ncal_{\gamma_t}$.
Recall that
$\Ncal_t(E)=\int_{\Sph\times E}x\dd\gamma_t(x,\omega)$, so
$|\Ncal_t|(\Sph)\le1$.  For $C_t\in\{A_t,B_t\}$, denote its continuous,
symmetric kernel by $K_C(t,\omega\cdot z)$.  We extend $C_t$ to the finite
vector measure $\Ncal_t$ by componentwise integration:
\[
 \begin{aligned}
 (C_t\Ncal_t)(z)
 &:=\int_{\Sph}K_C(t,z\cdot\omega)\dd\Ncal_t(\omega)\\
 &=\int_{\Sph\times\Sph}
 K_C(t,z\cdot\omega)x\dd\gamma_t(x,\omega).
 \end{aligned}
\]
Thus, for every $f\in L^\infty(\mu;\mathbb R^n)$,
\begin{equation}\label{eq:Fubini-swap}
 \int x\cdot(C_tf)(\omega)\dd\gamma_t(x,\omega)
 =\int f(z)\cdot(C_t\Ncal_t)(z)\dd\mu(z).
\end{equation}
Indeed, choose a Borel representative with $|f|\le\|f\|_\infty$
everywhere.  The scalar
integrand $K_C(t,\omega\cdot z)x\cdot f(z)$ is bounded in absolute value
by $\|K_C(t,\cdot)\|_\infty\|f\|_\infty$ and is therefore integrable
against the probability measure $\gamma_t\otimes\mu$.  Expanding the
kernel, applying Fubini's theorem, and using its symmetry gives
\begin{align*}
 \int x\cdot(C_tf)(\omega)\dd\gamma_t
 &=\int_{\Sph\times\Sph}\int_{\Sph}
 K_C(t,\omega\cdot z)x\cdot f(z)\dd\mu(z)\dd\gamma_t(x,\omega)\\
 &=\int_{\Sph}f(z)\cdot
 \left[\int_{\Sph\times\Sph}
 K_C(t,z\cdot\omega)x\dd\gamma_t(x,\omega)\right]\dd\mu(z)\\
 &=\int_{\Sph}f(z)\cdot(C_t\Ncal_t)(z)\dd\mu(z).
\end{align*}
Kernel integration gives continuous representatives of $C_tf$ and
$C_t\Ncal_t$; changing $f$ on a $\mu$-null set does not change either
side of \eqref{eq:Fubini-swap}.  In particular, this identity does not
require $m_t\ll\mu$.
Equation~\eqref{eq:Fubini-swap} pairs the rough trace $f$ with a smooth
kernel field against $\mu$.  No bound on the density of $m_t$ that is
uniform as $t$ approaches an endpoint is required.

By Theorem~\ref{thm:coupling-contraction}, Lemma~\ref{lem:kernel-monotone}, and \eqref{eq:first-multipliers},
\begin{equation}\label{eq:AB-contraction}
 \|A_t\Ncal_t\|_1\le a_1(t),
 \qquad
 \|B_t\Ncal_t\|_1\le b_1(t).
\end{equation}

\begin{proof}[Proof of Theorem~\ref{thm:sharp-bounds}]
Suppose first that $h$ preserves ends. Combining \eqref{eq:Poisson-pres}, \eqref{eq:slice-identity}, \eqref{eq:Fubini-swap}, and \eqref{eq:AB-contraction}, we get
\[
 R<R a_1(t)+b_1(t),\qquad r<t<1.
\]
Set $F(t)=R(a_1(t)-1)+b_1(t)$. Then $F(r)=0$ and $F(t)>0$ for $t>r$, so $F'(r+)\ge0$. Using \eqref{eq:inner-derivatives},
\[
 0\le -R\frac{n-1+r^n}{r(1-r^n)}+\frac{n}{1-r^n},
\]
which is equivalent to
\[
 R\le\Rplus(r).
\]

If $h$ reverses ends, the same argument using \eqref{eq:Poisson-rev} gives
\[
 R<a_1(t)+Rb_1(t),\qquad r<t<1.
\]
Let $G(t)=a_1(t)+R(b_1(t)-1)$. Then $G(1)=0$ and $G(t)>0$ for $t<1$, hence $G'(1-)\le0$. By \eqref{eq:outer-derivatives},
\[
 0\ge -\frac{nr^{n-1}}{1-r^n}
 +R\frac{1+(n-1)r^n}{1-r^n},
\]
so
\[
 R\le\Rminus(r).
\]

Finally,
\begin{equation}\label{eq:R-comparison}
 \Rplus(r)-\Rminus(r)=
 \frac{nr(1-r^2)}{(n-1+r^n)(1+(n-1)r^n)}
 \left(\sum_{k=0}^{n-2}r^{2k}-(n-1)r^{n-2}\right).
\end{equation}
The $n-1$ positive numbers $1,r^2,\ldots,r^{2n-4}$ have geometric mean
$r^{n-2}$.  Hence AM--GM applied to the bracket in
\eqref{eq:R-comparison} gives
\[
 \Rminus(r)<\Rplus(r).
\]
The radial maps in Section~\ref{sec:radial} attain both constants, so the
two bounds are sharp.
\end{proof}

\section{Critical rigidity}\label{sec:critical-rigidity}
We now prove the rigidity statement in
Theorem~\ref{thm:critical-rigidity}.  At a critical radius, the equality classification already proved in
Section~\ref{sec:coupling} reduces the problem to recovering the two
measurable boundary traces.

Define the endpoint transfer operators
\begin{align}
 \Tcal&=\lim_{t\downarrow r}\frac{B_t}{t-r}
       =\left.\partial_tB_t\right|_{t=r+},
       \label{eq:transfer-operator}\\
 \mathcal S&=\lim_{t\uparrow1}\frac{A_t}{1-t}
       =-\left.\partial_tA_t\right|_{t=1^-}.
       \label{eq:reverse-transfer-operator}
\end{align}
For $n\ge3$, and also for $n=2$ and $\ell\ge1$, their multipliers
on degree-$\ell$ spherical harmonics are
\begin{align}
 \tau_\ell&=b_\ell'(r)
 =\frac{(2\ell+n-2)r^{\ell-1}}{1-r^{2\ell+n-2}},
 &\tau_1&=\frac{n}{1-r^n},
 \label{eq:tauell}\\
 \sigma_\ell&=-a_\ell'(1)
 =\frac{(2\ell+n-2)r^{\ell+n-2}}{1-r^{2\ell+n-2}},
 &\sigma_1&=\frac{nr^{n-1}}{1-r^n}.
 \label{eq:reverse-sigma}
\end{align}

When $n=2$, the remaining multipliers follow from
\eqref{eq:zero-multipliers}:
\[
 \tau_0=\frac1{r\log(1/r)},\qquad
 \sigma_0=\frac1{\log(1/r)}.
\]

\begin{lemma}\label{lem:endpoint-transfer}
Let $n\ge2$.  Let $K_{B,t}$ and $K_{A,t}$ denote the kernels of $B_t$ and $A_t$ with
respect to $\mu$.  There are smooth zonal kernels $K_T$ and $K_S$ such
that, for every integer $m\ge0$,
\begin{align}
 \left\|\frac{K_{B,t}}{t-r}-K_T\right\|_{C^m(\Sph\times\Sph)}
 &\longrightarrow0 &&(t\downarrow r),
 \label{eq:kernel-limit}\\
 \left\|\frac{K_{A,t}}{1-t}-K_S\right\|_{C^m(\Sph\times\Sph)}
 &\longrightarrow0 &&(t\uparrow1).
 \label{eq:reverse-kernel-Cinfty}
\end{align}
The integral operators associated with $K_T$ and $K_S$ are $\Tcal$ and
$\mathcal S$, respectively.  Both kernels are positive on $[-1,1]$ and
strictly increasing in the scalar product.
\end{lemma}

\begin{proof}
Let $D=\A(r,1)$ and $d(y)=\operatorname{dist}(y,\partial D)$.
By the Poisson kernel formula of Engli\v{s} \cite[Theorem~1, formula~(8)]{Englis},
$p_D(y,\zeta)/d(y)$ extends smoothly up to the boundary away from
$y=\zeta$. In the two receiving collars, $d(tz)$ is $t-r$ and $1-t$,
respectively. Since the two boundary spheres are disjoint,
\eqref{eq:KA-KB-Poisson} gives both $C^m$ limits. The Poisson density
vanishes on the receiving sphere, so these limits are
\[
 K_T(z\cdot\omega)=|\Sph|\,\partial_{\mathrm{in}}p_D(rz,\omega),
 \qquad
 K_S(z\cdot\omega)=|\Sph|r^{n-1}\,\partial_{\mathrm{in}}p_D(z,r\omega),
\]
where $\partial_{\mathrm{in}}$ is the inward normal derivative in the
first variable. The resulting zonal kernels represent $\Tcal$ and
$\mathcal S$ by \eqref{eq:transfer-operator}--\eqref{eq:reverse-transfer-operator}.

For fixed $b\in\{r,1\}$ and $\omega\in\Sph$, the function
$p_D(\cdot,b\omega)$ is positive in $D$ and vanishes on $\partial B_{r/b}$.
The Hopf lemma gives $K_T,K_S>0$ by the formulas above.

Fix $z,\omega_1,\omega_2\in\Sph$ with $z\cdot\omega_1>z\cdot\omega_2$.
Set
\[
 H=\{y:y\cdot(\omega_1-\omega_2)>0\},\qquad
 w_b(y)=p_D(y,b\omega_1)-p_D(y,b\omega_2),\quad b\in\{r,1\}.
\]
The Poisson density is harmonic in its first variable and vanishes on
the opposite sphere. Together with the reflection comparison in the
proof of Lemma~\ref{lem:kernel-monotone}, this gives
\[
 \Delta w_b=0,\quad w_b>0\quad\text{in }D\cap H,
 \qquad w_b=0\quad\text{on }\partial B_{r/b}\cap H.
\]
Since $(r/b)z\in\partial B_{r/b}\cap H$, the inward normal to $D\cap H$
there agrees with that of $D$. The Hopf lemma yields
$\partial_{\mathrm{in}}w_b((r/b)z)>0$. Hence
\[
 \begin{aligned}
 K_T(z\cdot\omega_1)-K_T(z\cdot\omega_2)
 &=|\Sph|\,\partial_{\mathrm{in}}w_1(rz)>0,\\
 K_S(z\cdot\omega_1)-K_S(z\cdot\omega_2)
 &=|\Sph|r^{n-1}\,\partial_{\mathrm{in}}w_r(z)>0.
 \end{aligned}
\]
\end{proof}

\begin{lemma}\label{lem:uniform-varying}
Let $n\ge2$.  Let $K_j,K\in C(\Sph\times\Sph)$ with $K_j\to K$ uniformly, and let
$\gamma_j\rightharpoonup\gamma$ be probability measures on
$\Sph\times\Sph$.  Then
\begin{equation}\label{eq:abstract-uniform-varying}
 \int xK_j(z,\omega)\dd\gamma_j(x,\omega)
 \longrightarrow
 \int xK(z,\omega)\dd\gamma(x,\omega)
\end{equation}
uniformly for $z\in\Sph$.
\end{lemma}

\begin{proof}
The contribution of $K_j-K$ is bounded uniformly in $z$ by
$\|K_j-K\|_\infty$.  For the remaining term, put
\[
 F_z(x,\omega)=K(z,\omega)x,
 \qquad
 L_j(F)=\int F\dd\gamma_j-\int F\dd\gamma.
\]
Weak convergence gives $L_j(F)\to0$ for every continuous $F$, while
\[
 |L_j(F)-L_j(G)|\le2\|F-G\|_\infty.
\]
The map $z\mapsto F_z$ is continuous from the compact sphere into
$C(\Sph\times\Sph;\mathbb R^n)$, so $\{F_z:z\in\Sph\}$ is compact in the
uniform norm.  A finite $\varepsilon$-net and the last two facts give
$\sup_z|L_j(F_z)|\to0$, which proves
\eqref{eq:abstract-uniform-varying}.
\end{proof}

\subsection{The end-preserving critical case}
\label{sec:rigidity-preserving}

Assume $n\ge3$ and
\[
 R=\Rstar:=\Rplus(r)=\frac{nr}{n-1+r^n}.
\]
The strict comparison \eqref{eq:R-comparison} shows that a critical map cannot reverse ends. Hence
\begin{equation}\label{eq:critical-Poisson}
 h(t,\cdot)=\Rstar A_t\phi+B_t\psi.
\end{equation}
Define
\[
 I_A(t)=\ip{\phi}{A_t\Ncal_t},
 \qquad
 I_B(t)=\ip{\psi}{B_t\Ncal_t},
\]
where the brackets denote the $L^2(\mu)$ dual pairing, and define
\[
 G_t=\int\bigl(|h(t,\omega)|-\Rstar\bigr)\dd m_t(\omega)>0.
\]
By \eqref{eq:slice-identity},
\[
 \Rstar I_A(t)+I_B(t)=\Rstar+G_t.
\]
On the other hand, $|\phi|=|\psi|=1$ a.e. and \eqref{eq:AB-contraction} give
$I_A(t)\le a_1(t)$ and $I_B(t)\le b_1(t)$. Therefore
\begin{equation}\label{eq:finite-defect}
 D_n(t):=\Rstar a_1(t)+b_1(t)-\Rstar
 =\Rstar(a_1(t)-I_A(t))+(b_1(t)-I_B(t))+G_t,
\end{equation}
with all three terms on the right nonnegative.

Substituting \eqref{eq:first-multipliers} and \eqref{eq:Rpm} into the
left-hand side of \eqref{eq:finite-defect} gives
\begin{equation}\label{eq:Dn}
 D_n(t)=\frac{(n-1)t+r^nt^{1-n}-nr}{n-1+r^n}.
\end{equation}
If $u=t/r$, then
\begin{equation}\label{eq:Dn-factor}
 D_n(t)=\frac{r(u-1)^2}{(n-1+r^n)u^{n-1}}
 \sum_{j=0}^{n-2}(j+1)u^j.
\end{equation}
Indeed,
\[
 (n-1)u^n-nu^{n-1}+1=(u-1)^2\sum_{j=0}^{n-2}(j+1)u^j.
\]
Consequently
\begin{equation}\label{eq:Dn-order}
 D_n(t)=\frac{n(n-1)}{2r(n-1+r^n)}(t-r)^2+O((t-r)^3),
\end{equation}
and, in particular,
\begin{equation}\label{eq:B-defect}
 0\le b_1(t)-I_B(t)\le D_n(t)=O((t-r)^2).
\end{equation}

To identify the outer trace, choose any sequence $t_j\downarrow r$.
Compactness of probability measures on $\Sph\times\Sph$ gives, after
passing to a subsequence,
\begin{equation}\label{eq:gamma-weak}
 \gamma_{t_j}\rightharpoonup\gamma.
\end{equation}
The first marginal of $\gamma$ is still $\mu$.  Set
\[
 U_j(z)=\frac{(B_{t_j}\Ncal_{t_j})(z)}{t_j-r}.
\]
By \eqref{eq:kernel-limit} and
Lemma~\ref{lem:uniform-varying}, applied with
\[
 K_j(z,\omega)=\frac{K_{B,t_j}(z,\omega)}{t_j-r},
 \qquad K=K_T,
\]
we have
\begin{equation}\label{eq:uniform-transfer}
 U_j\longrightarrow\Tcal\Ncal_\gamma
 \quad\text{uniformly on }\Sph.
\end{equation}

Divide \eqref{eq:B-defect} by $t_j-r$:
\[
 0\le \frac{b_1(t_j)}{t_j-r}-\ip{\psi}{U_j}
 \le \frac{D_n(t_j)}{t_j-r}\to0.
\]
Recalling $\tau_1=b_1'(r)=n/(1-r^n)$ from \eqref{eq:tauell},
Lemma~\ref{lem:uniform-varying} gives
\begin{equation}\label{eq:transfer-equality}
 \ip{\psi}{\Tcal\Ncal_\gamma}=\tau_1.
\end{equation}
Since $|\psi|=1$ a.e.,
\[
 \ip{\psi}{\Tcal\Ncal_\gamma}
 \le\|\Tcal\Ncal_\gamma\|_1\le\tau_1.
\]
Thus equality holds in the transfer coupling contraction. By
Lemma~\ref{lem:endpoint-transfer} and
Theorem~\ref{thm:coupling-contraction}, there exists $M\in O(n)$ such that
$\gamma=(x,Mx)_\#\mu$. Consequently
\[
 \dd\Ncal_\gamma(\omega)=M^T\omega\dd\mu(\omega),
 \qquad
 \Tcal\Ncal_\gamma(z)=\tau_1M^Tz.
\]
Substituting this into \eqref{eq:transfer-equality} gives
\[
 \int_{\Sph}\psi(z)\cdot M^Tz\dd\mu(z)=1.
\]
The integrand is at most one pointwise, so
\begin{equation}\label{eq:outertrace}
 \psi(z)=M^Tz\quad\text{for }\mu\text{-a.e. }z.
\end{equation}
The following lemma identifies the other trace for both boundary correspondences.

\begin{lemma}\label{lem:endpoint-graph}
Let $n\ge2$ and $0<r,R<1$. Suppose $h$ has one of the representations
\eqref{eq:Poisson-pres}--\eqref{eq:Poisson-rev}, with
$\phi,\psi\in L^\infty(\Sph;\Sph)$ and $|h|\ge R$.
Set $t_*=r$ in the first case and $t_*=1$ in the second.
If $t_j\to t_*$ and $\gamma_{t_j}\rightharpoonup
(\operatorname{id},M)_\#\mu$ for some $M\in O(n)$, where the couplings
are supported on $x=q_{t_j}(\omega)$, then $\phi$ in the first case and
$\psi$ in the second equals $M^T\omega$ almost everywhere.
\end{lemma}

\begin{proof}
Write $g$ for the normalized trace at $t_*$. By
Lemma~\ref{lem:endpoint-trace-regularity}, we may take $g$ continuous,
with $q_{t_j}\to g$ uniformly. The function $(x,\omega)\mapsto |x-g(\omega)|$
is continuous, and the graph relation gives
\[
 \int_{\Sph}|x-g(Mx)|\dd\mu(x)
 =\lim_{j\to\infty}\int |x-g(\omega)|\dd\gamma_{t_j}(x,\omega)
 \le\lim_{j\to\infty}\|q_{t_j}-g\|_\infty=0.
\]
Thus $g(Mx)=x$ almost everywhere. Since $M$ preserves $\mu$, this gives
$g(\omega)=M^T\omega$ almost everywhere.
\end{proof}

In the preserving case, Lemma~\ref{lem:endpoint-graph} and
\eqref{eq:outertrace} give $\phi(\omega)=\psi(\omega)=Q\omega$ almost
everywhere, with $Q=M^T$.
Substitution into the Poisson representation \eqref{eq:critical-Poisson}
gives
\[
 h(t,\omega)=\bigl(\Rstar a_1(t)+b_1(t)\bigr)Q\omega.
\]
Substitution of \eqref{eq:first-multipliers} and
$\Rstar=nr/(n-1+r^n)$ yields
\[
 \Rstar a_1(t)+b_1(t)=
 \frac{(n-1)t+r^nt^{1-n}}{n-1+r^n}=H_{n,r}(t).
\]
This is exactly \eqref{eq:critical-map}.  The radial computation in
Section~\ref{sec:radial} already shows that this map is a homeomorphism of
the claimed annuli.  This proves Theorem~\ref{thm:critical-rigidity}\rm(i).

\subsection{The end-reversing critical case}
\label{sec:rigidity-reversing}
We now assume that the ends are interchanged and
\begin{equation}\label{eq:reverse-critical-assumption}
 R=\Rminus(r)=\frac{nr^{n-1}}{1+(n-1)r^n}.
\end{equation}
The Poisson representation is
\begin{equation}\label{eq:reverse-Poisson}
 h(t,\cdot)=A_t\phi+RB_t\psi,
 \qquad
 \phi,\psi\in L^\infty(\Sph;\Sph).
\end{equation}

For the finite-slice coupling of Lemma~\ref{lem:slice}, set
\[
 I_A^-(t)=\ip{\phi}{A_t\Ncal_t},
 \qquad
 I_B^-(t)=\ip{\psi}{B_t\Ncal_t},
\]
and
\[
 G_t^-=
 \int\bigl(|h(t,\omega)|-R\bigr)\dd m_t(\omega)>0.
\]
Equations \eqref{eq:slice-identity}, \eqref{eq:Fubini-swap}, and
\eqref{eq:reverse-Poisson} give
\[
 I_A^-(t)+RI_B^-(t)=R+G_t^-.
\]
By \eqref{eq:AB-contraction} and the sphere-valued traces,
$I_A^-(t)\le a_1(t)$ and $I_B^-(t)\le b_1(t)$.  Hence
\begin{equation}\label{eq:reverse-defect-split}
 D_n^-(t):=a_1(t)+Rb_1(t)-R
 =\bigl(a_1(t)-I_A^-(t)\bigr)
 +R\bigl(b_1(t)-I_B^-(t)\bigr)+G_t^-.
\end{equation}
All terms on the right are nonnegative.  Direct substitution of
\eqref{eq:first-multipliers} and \eqref{eq:reverse-critical-assumption}
gives the exact factorization
\begin{align}
 D_n^-(t)
 &=\frac{r^{n-1}\bigl((n-1)t+t^{1-n}-n\bigr)}
 {1+(n-1)r^n}\notag\\
 &=\frac{r^{n-1}(1-t)^2}
 {[1+(n-1)r^n]t^{n-1}}
 \sum_{j=0}^{n-2}(j+1)t^j.
 \label{eq:reverse-defect-factor}
\end{align}
Consequently
\begin{equation}\label{eq:reverse-A-defect}
 0\le a_1(t)-I_A^-(t)
 \le D_n^-(t)=O((1-t)^2)
 \qquad(t\uparrow1).
\end{equation}

We now prove the trace estimate used in Lemma~\ref{lem:endpoint-graph}.

\begin{lemma}\label{lem:endpoint-trace-regularity}
Let $n\ge2$ and $0<r,R<1$. Suppose $h$ has one of the representations
\eqref{eq:Poisson-pres}--\eqref{eq:Poisson-rev}, with
$\phi,\psi\in L^\infty(\Sph;\Sph)$ and $|h|\ge R$ in $\A(r,1)$.
Then $\phi$ in the first case and $\psi$ in the second has a
$C^{0,1/2}(\Sph;\Sph)$ representative, and $q_t$ converges uniformly
to that representative as $t\downarrow r$ or $t\uparrow1$, respectively.
\end{lemma}

\begin{proof}
Let $t_*=r$ in the first case and $t_*=1$ in the second. Write
$d(y)=\bigl||y|-t_*\bigr|$ and let $g$ be the normalized trace at $t_*$.
The Poisson operator from the opposite sphere has mass $m(t)=b_0(t)$
or $m(t)=a_0(t)$, respectively. By positivity and
\eqref{eq:zero-multipliers},
\[
 R\le |h(t,\omega)|\le R(1-m(t))+m(t)
 \le R+C|t-t_*|
\]
near $t_*$. For $y$ sufficiently close to the sphere $|y|=t_*$, put
$\rho=d(y)/2$, so that $B_\rho(y)\subset\A(r,1)$.
On $\partial B_\rho(y)$ we have $d(z)\le 3d(y)/2$. The mean-value
formula \cite[Theorem~2.1]{GT} gives
\[
 \frac{1}{|\partial B_\rho|}
 \int_{\partial B_\rho(y)}|h(z)-h(y)|^2\dd S(z)
 =\frac{1}{|\partial B_\rho|}
 \int_{\partial B_\rho(y)}|h(z)|^2\dd S(z)-|h(y)|^2
 \le Cd(y).
\]
Applying the ball mean-value formula to $Dh$ and using the divergence
theorem gives
\[
 |Dh(y)|\le\frac{n}{\rho}
 \left(\frac{1}{|\partial B_\rho|}
 \int_{\partial B_\rho(y)}|h(z)-h(y)|^2\dd S(z)\right)^{1/2}
 \le C d(y)^{-1/2}.
\]

Radial integration gives a uniform limit of $h(t,\cdot)$ at $t_*$ with
error at most $C|t-t_*|^{1/2}$. The $L^2$ trace identifies this limit
with $Rg$ almost everywhere. To compare the limit at $\omega$ and $\eta$,
let $\delta=d_{\Sph}(\omega,\eta)$ be small and choose $t$ on the
annular side with $|t-t_*|=\delta$. The gradient estimate along the
shorter spherical arc gives
$|h(t,\omega)-h(t,\eta)|\le C\delta^{1/2}$.
Together with the two radial errors this yields the claimed H\"older
regularity. The representative is sphere-valued because $|g|=1$
almost everywhere. Finally, normalization is Lipschitz on
$\{v:|v|\ge R\}$, so $q_t\to g$ uniformly.
\end{proof}

To pass to the endpoint in \eqref{eq:reverse-A-defect}, we use the
inner-to-outer operator $\mathcal S$ from
\eqref{eq:reverse-transfer-operator}.  Lemma~\ref{lem:endpoint-transfer}
supplies its smooth, positive, strictly increasing kernel and the
$C^\infty$ convergence \eqref{eq:reverse-kernel-Cinfty}.

We now identify both traces.  Choose any sequence $t_j\uparrow1$ and, after passing to a subsequence,
write
\[
 \gamma_{t_j}\rightharpoonup\gamma.
\]
The first marginal of $\gamma$ is still $\mu$.  Put $s_j=1-t_j$.
The convergence \eqref{eq:reverse-kernel-Cinfty} and
Lemma~\ref{lem:uniform-varying} give
\begin{equation}\label{eq:reverse-uniform-varying}
 \frac{A_{t_j}\Ncal_{t_j}}{s_j}\longrightarrow
 \mathcal S\Ncal_\gamma
 \quad\text{uniformly on }\Sph.
\end{equation}

Divide \eqref{eq:reverse-A-defect} by $s_j$.  Since
$a_1(t_j)/s_j\to\sigma_1$, \eqref{eq:reverse-uniform-varying} and
$\phi\in L^\infty$ imply
\begin{equation}\label{eq:reverse-transfer-equality}
 \ip{\phi}{\mathcal S\Ncal_\gamma}=\sigma_1.
\end{equation}
By Lemma~\ref{lem:endpoint-transfer} and
Theorem~\ref{thm:coupling-contraction},
\[
 \ip{\phi}{\mathcal S\Ncal_\gamma}
 \le\|\mathcal S\Ncal_\gamma\|_1\le\sigma_1.
\]
Thus equality holds throughout.  There exists $M\in O(n)$ such that
\[
 \gamma=(x,Mx)_\#\mu,
 \qquad
 \dd\Ncal_\gamma(\omega)=M^T\omega\dd\mu(\omega),
 \qquad
 \mathcal S\Ncal_\gamma(z)=\sigma_1M^Tz.
\]
Substitution into \eqref{eq:reverse-transfer-equality}, together with
$|\phi|=1$, gives
\begin{equation}\label{eq:reverse-innertrace}
 \phi(z)=M^Tz
 \quad\text{for }\mu\text{-almost every }z.
\end{equation}

Lemma~\ref{lem:endpoint-graph} gives $\psi(\omega)=M^T\omega$
almost everywhere. Together with \eqref{eq:reverse-innertrace}, the
Poisson representation \eqref{eq:reverse-Poisson} now gives, for $Q=M^T$,
\[
 h(t,\omega)=\bigl(a_1(t)+\Rminus(r)b_1(t)\bigr)Q\omega.
\]
Substitution of \eqref{eq:first-multipliers} yields
\[
 a_1(t)+\Rminus(r)b_1(t)
 =\frac{r^{n-1}\bigl((n-1)t+t^{1-n}\bigr)}{1+(n-1)r^n}.
\]
This is \eqref{eq:reverse-critical-map}.  The monotonicity computation in
\eqref{eq:Hminus} shows that it is an end-reversing homeomorphism.  This
proves Theorem~\ref{thm:critical-rigidity}\rm(ii).

\section{Nonzero degree and spherical constraint maps}
\label{sec:nonzero-degree}

We now prove Theorem~\ref{thm:nonzero-degree}, stated in Section~\ref{sec:consequences}, and then derive the planar winding-number bound and the spherical contact-radius estimate.

\begin{proof}[Proof of Theorem~\ref{thm:nonzero-degree}]
Since $h$ is continuous and never vanishes, the maps $q_t$ are homotopic
as $t$ varies.  Their common degree is nonzero, so each $q_t$ is onto.
The Borel selection argument of Lemma~\ref{lem:slice} therefore gives a
right inverse $\sigma_t$ and a coupling
$\gamma_t=(x,\sigma_t(x))_\#\mu$ with first marginal $\mu$.
Writing $m_t=(\sigma_t)_\#\mu$ and $\Ncal_t=\Ncal_{\gamma_t}$, we have
\begin{equation}\label{eq:nonzero-degree-slice}
 \int x\cdot h(t,\omega)\dd\gamma_t(x,\omega)
 =\int |h(t,\omega)|\dd m_t(\omega)\ge R.
\end{equation}

Lemma~\ref{lem:traces} gives $L^2$ traces $g_r,g_1$ and the annular
Poisson representation.  Since the modulus is Lipschitz, the endpoint
conditions imply $|g_r|=R$, $|g_1|=1$ almost everywhere in case
\rm(i), and $|g_r|=1$, $|g_1|=R$ in case \rm(ii).
Thus \eqref{eq:Poisson-pres}--\eqref{eq:Poisson-rev} hold with
measurable sphere-valued $\phi,\psi$.  Combining
\eqref{eq:nonzero-degree-slice}, \eqref{eq:Fubini-swap}, and
\eqref{eq:AB-contraction} gives
\[
 \begin{aligned}
 R&\le R a_1(t)+b_1(t) &&\text{in case \rm(i)},\\
 R&\le a_1(t)+Rb_1(t) &&\text{in case \rm(ii)}.
 \end{aligned}
\]
The one-sided derivatives at $r$ and $1$, respectively, give the stated
bounds exactly as in Section~\ref{sec:sharp-bounds}.

We show why the equality arguments also survive the weaker hypotheses.
In case \rm(i), suppose $R=\Rplus(r)$.  The exact defect identity
\eqref{eq:finite-defect} still holds, now with
\[
 G_t=\int\bigl(|h(t,\omega)|-R\bigr)\dd m_t(\omega)\ge0.
\]
All its terms remain nonnegative, and \eqref{eq:Dn-order} still gives
$D_n(t)=O((t-r)^2)$.  Passing to a weak limit of $\gamma_t$ as
$t\downarrow r$ and dividing the $B_t$ defect by $t-r$,
Lemmas~\ref{lem:endpoint-transfer} and \ref{lem:uniform-varying} give
\[
 \ip{\psi}{\Tcal\Ncal_\gamma}=\tau_1.
\]
Theorem~\ref{thm:coupling-contraction} and equality in the pairing with
the unit trace yield
\[
 \gamma=(x,Mx)_\#\mu,
 \qquad \psi(\omega)=M^T\omega\quad\text{a.e.}
\]
for some $M\in O(n)$. Lemma~\ref{lem:endpoint-graph} gives
$\phi(\omega)=M^T\omega$ almost everywhere. Substitution of the two
traces gives the formula in \rm(i).

In case \rm(ii), suppose $R=\Rminus(r)$.  The identity
\eqref{eq:reverse-defect-split} likewise holds with $G_t^-\ge0$, and
\eqref{eq:reverse-defect-factor} gives $D_n^-(t)=O((1-t)^2)$.
Passing to a weak limit as $t\uparrow1$ and dividing the $A_t$ defect
by $1-t$ yields
\[
 \ip{\phi}{\mathcal S\Ncal_\gamma}=\sigma_1.
\]
The transfer-kernel equality classification gives
$\gamma=(x,Mx)_\#\mu$ and $\phi(\omega)=M^T\omega$ almost everywhere.
Lemma~\ref{lem:endpoint-graph} gives $\psi(\omega)=M^T\omega$
almost everywhere, and substitution gives the formula in \rm(ii).
Both arguments apply for every $n\ge2$.

The radial maps of Section~\ref{sec:radial} attain both bounds and have
direction maps of degree $\det Q\in\{-1,1\}$.  This proves sharpness,
the converse equality assertions, and strictness when the absolute
degree is at least two.
\end{proof}

The endpoint modulus conditions in Theorem~\ref{thm:nonzero-degree}
are not required for the following planar bound.  It proves the
inequality in \cite[Conjecture~1.8]{IKO}; the nonstrict modulus constraint
also permits the double-cover equality profile described there.

\begin{corollary}[Planar maps of nonzero winding number]
\label{cor:IKO-winding}
Let $0<r,\rho<1$ and let $h:\A(r,1)\to\mathbb R^2$ be harmonic with
$\rho\le |h|\le1$.  Suppose that $h/|h|$ has nonzero degree on one,
and hence every, interior circle.  Then
\begin{equation}\label{eq:planar-winding-bound}
 \rho\le\frac{2\sqrt r}{1+r}.
\end{equation}
Every equality map has the form
\begin{equation}\label{eq:planar-double-cover}
 h(t,\omega)=\frac{t+r/t}{1+r}\,Q\omega,
 \qquad Q\in O(2).
\end{equation}
Conversely, these maps attain the bound when $\rho=2\sqrt r/(1+r)$.
No endpoint modulus condition, injectivity, or surjectivity is assumed.
\end{corollary}

\begin{proof}
Lemma~\ref{lem:traces} gives vector $L^2$ traces $g_r,g_1$ and the
representation $h=A_tg_r+B_tg_1$.  Strong convergence and $|h|\le1$
imply $|g_r|,|g_1|\le1$ almost everywhere.  The nonzero degree makes
each direction map onto.  The selection argument of Lemma~\ref{lem:slice}
therefore gives $\gamma_t$, $m_t$, and $\Ncal_t$ as before.
By \eqref{eq:Fubini-swap} and Theorem~\ref{thm:coupling-contraction},
\[
 \begin{aligned}
 \rho
 &\le\int |h(t,\omega)|\dd m_t(\omega)\\
 &=\ip{g_r}{A_t\Ncal_t}+\ip{g_1}{B_t\Ncal_t}\\
 &\le\|A_t\Ncal_t\|_1+\|B_t\Ncal_t\|_1\\
 &\le a_1(t)+b_1(t)=\frac{t+r/t}{1+r}.
 \end{aligned}
\]
Taking $t=\sqrt r$ proves \eqref{eq:planar-winding-bound}.

If equality holds, all the inequalities above are equalities at this
fixed interior radius.  Each of the two contraction deficits and each
of the two pairing deficits is nonnegative, so all four vanish.
The strict monotonicity of the Poisson kernels and
Theorem~\ref{thm:coupling-contraction} give
\[
 \gamma_t=(x,Mx)_\#\mu,\qquad
 A_t\Ncal_t(z)=a_1(t)M^Tz,\qquad
 B_t\Ncal_t(z)=b_1(t)M^Tz
\]
for some $M\in O(2)$.  Since $a_1(t),b_1(t)>0$ and
$|g_r|,|g_1|\le1$, equality in the two pairings forces
$g_r(z)=g_1(z)=M^Tz$ almost everywhere.  The Poisson representation
then gives \eqref{eq:planar-double-cover} with $Q=M^T$.
Conversely, this profile is harmonic, has degree $\det Q\ne0$, and its
modulus has minimum $2\sqrt r/(1+r)$ at $t=\sqrt r$ and tends to $1$
at both endpoints.  It therefore attains the bound when
$\rho=2\sqrt r/(1+r)$.
\end{proof}

Rescaling \eqref{eq:planar-winding-bound}, every harmonic map
$h:\A(r_0,r_1)\to\A(R_0,R_1)\subset\mathbb R^2$ of nonzero winding
number satisfies
\[
 \frac{R_1}{R_0}\ge
 \frac12\left(\sqrt{\frac{r_1}{r_0}}+
                 \sqrt{\frac{r_0}{r_1}}\right),
\]
which is precisely \cite[(1.17)]{IKO}.  For a map into the open target
annulus, the lower bound on each compact interior circle is strict,
so this inequality is strict.  The same constant is best possible:
the profile \eqref{eq:planar-double-cover} maps into $\A(\rho,1)$
for every $\rho<2\sqrt r/(1+r)$ and attains the limiting constant
when contact with the inner target circle is allowed.

The exterior-ball constraint problem connects this result to geometric
obstacle and elliptic free-boundary problems; see
\cite{LinThesis,LinFreeBoundaryLectures} for background and
\cite{FigalliGuerraKimShahgholianSurvey,FigalliKimShahgholian}
for the constraint-map framework.  The following consequence concerns the
harmonic exterior phase of a spherical contact core and requires no
injectivity assumption.

\begin{corollary}[Spherical contact radius]\label{cor:constraint-radius}
Let $n\ge3$, $0<a,\rho<1$, and let $\rho_a\in(0,1)$ be the unique number
such that
\[
 a=\Rplus(\rho_a)=\frac{n\rho_a}{n-1+\rho_a^n}.
\]
Suppose $h:\A(\rho,1)\to\mathbb R^n$ is bounded and coordinatewise
harmonic, $|h|\ge a$, and its $L^2$ boundary traces are
$a\phi$ on the inner sphere and $\omega\mapsto\omega$ on the outer
sphere, where $|\phi|=1$ almost everywhere.  Then $\rho\ge\rho_a$.
If $\rho=\rho_a$, then
\[
 h(t,\omega)=H_{n,\rho_a}(t)\omega,
 \qquad \rho_a<t<1.
\]
\end{corollary}

\begin{proof}
Use the annular Poisson operators with inner radius $\rho$.
The representation is $h(t,\omega)=aA_t\phi(\omega)+b_1(t)\omega$.
Positivity gives $\|A_t\phi\|_\infty\le a_0(t)$, while
$a_0(t)\to0$ and $b_1(t)\to1$ as $t\uparrow1$.
Thus $h(t,\cdot)\to\operatorname{id}_{\Sph}$ uniformly, and the
direction map $q_t$ has degree one for $t$ sufficiently close to $1$.
Theorem~\ref{thm:nonzero-degree}\rm(i) gives $a\le\Rplus(\rho)$.
Since $\Rplus$ is strictly increasing from $0$ to $1$ on $(0,1)$,
we obtain $\rho\ge\rho_a$.  At equality, the same theorem gives
$h(t,\omega)=H_{n,\rho_a}(t)Q\omega$, and the identity outer trace
forces $Q=I$.
\end{proof}

At the critical radius, $H_{n,\rho_a}(\rho_a)=a$ and
$H_{n,\rho_a}'(\rho_a)=0$.  Gluing this exterior profile to
$a\omega$ on $0<t\le\rho_a$ gives the canonical radial
constraint-map vortex \cite{DengLiLiuMa}.  Corollary~\ref{cor:constraint-radius}
identifies the sharp contact-radius obstruction and the exterior equality
profile under the stated harmonic-phase hypotheses; it does not assert
minimality or uniqueness of the map inside the contact core.

\section*{Acknowledgments}
\noindent We are deeply grateful to Prof.~Fang-Hua Lin for his guidance and encouragement over many years. In his courses on free-boundary problems, he repeatedly emphasized the connection between the Nitsche problem and the constraint-map problem, as reflected in Corollary~\ref{cor:constraint-radius}. Our interest in both problems grew out of his lectures and guidance. This work was supported by the National Key R\&D Program of China (2025YFA1017603).

\medskip
\noindent We used ChatGPT (OpenAI) for language editing
and assistance with exploring proof strategies and calculations, particularly
in developing the proof of
Theorem~\ref{thm:coupling-contraction}.  We checked the
arguments and take responsibility for the manuscript.

\bigskip
\noindent
(Bin Deng) School of Mathematics and Statistics, Wuhan University, Wuhan,
Hubei 430072, P.R. China.\\
Email address: \href{mailto:dbmath@whu.edu.cn}{dbmath@whu.edu.cn}

\par\smallskip
\noindent
(Jiahuan Li) School of Mathematical Sciences, University of Science and
Technology of China, Hefei, Anhui 230026, P.R. China.\\
Email address: \href{mailto:jiahuan@mail.ustc.edu.cn}{jiahuan@mail.ustc.edu.cn}

\par\smallskip
\noindent
(Yilu Liu) School of Mathematical Sciences, University of Science and
Technology of China, Hefei, Anhui 230026, P.R. China.\\
Email address: \href{mailto:liuylgeoanaly@mail.ustc.edu.cn}{liuylgeoanaly@mail.ustc.edu.cn}

\par\smallskip
\noindent
(Xi-Nan Ma) School of Mathematical Sciences, University of Science and
Technology of China, Hefei, Anhui 230026, P.R. China.\\
Email address: \href{mailto:xinan@ustc.edu.cn}{xinan@ustc.edu.cn}

\end{document}